\documentclass[11pt,reqno]{amsart}
\usepackage{amsthm,amsmath,amssymb,mathrsfs}

\theoremstyle{plain}
\newtheorem{theorem}{Theorem} [section]
\newtheorem{corollary}[theorem]{Corollary}
\newtheorem{lemma}[theorem]{Lemma}
\newtheorem{proposition}[theorem]{Proposition}

\theoremstyle{definition}

\newtheorem{remark}[theorem]{Remark}

\def\N{{\mathbb{N}}}
\def\R{{\mathbb{R}}}
\def\Z{{\mathbb{Z}}}

\def\A{{\mathcal{A}}}

\def\clspan{{\overline{\mathrm{span}}}}

\newcommand{\p}{\varphi}
\newcommand{\w}{\omega}

\newcommand{\g}{\gamma}
\newcommand{\dd}{\delta}
\newcommand{\G}{\Gamma}
\newcommand{\W}{\Omega}

\newcommand{\T}{\mathcal{T}}
\newcommand{\s}{\sigma}
\newcommand{\SR}{\mathcal{D}}
\newcommand{\SN}{\mathcal{N}}

\newtheorem{thm*}{Teorema}[section]

\newcommand{\CHI}{\hbox{\raise .4ex \hbox{$\chi$}}}

\newcommand{\set}[1]{\{#1\}}

\newcommand{\dotoplus}{\operatornamewithlimits{\dot{\oplus}}}
\newcommand{\dotbigoplus}{\operatornamewithlimits{\dot{\bigoplus}}}

\def\subset{\subseteq}
\def\supp{{\textrm{supp}}}

\def\clspan{{\overline{\mathrm{span}}}}

\def\T{{\mathcal{T}}}

\def\EQ{\, = \,}

\begin{document}

\title{Extra Invariance of shift-invariant spaces on LCA groups}
\author[M. Anastasio, C. Cabrelli, and V. Paternostro]{M. Anastasio,
C. Cabrelli, and V. Paternostro}

\address{\textrm{(M. Anastasio)}
Departamento de Matem\'atica,
Facultad de Ciencias Exac\-tas y Naturales,
Universidad de Buenos Aires, Ciudad Universitaria, Pabell\'on I,
1428 Buenos Aires, Argentina and
CONICET, Consejo Nacional de Investigaciones
Cient\'ificas y T\'ecnicas, Argentina}
\email{manastas@dm.uba.ar}

\address{\textrm{(C. Cabrelli)}
Departamento de Matem\'atica,
Facultad de Ciencias Exac\-tas y Naturales,
Universidad de Buenos Aires, Ciudad Universitaria, Pabell\'on I,
1428 Buenos Aires, Argentina and
CONICET, Consejo Nacional de Investigaciones
Cient\'ificas y T\'ecnicas, Argentina}
\email{cabrelli@dm.uba.ar}

\address{\textrm{(V.Paternostro)}
Departamento de Matem\'atica,
Facultad de Ciencias Exac\-tas y Naturales,
Universidad de Buenos Aires, Ciudad Universitaria, Pabell\'on I,
1428 Buenos Aires, Argentina and
CONICET, Consejo Nacional de Investigaciones
Cient\'ificas y T\'ecnicas, Argentina}
\email{vpater@dm.uba.ar}
 \thanks{The research of
the authors is partially supported by
Grants: ANPCyT, PICT 2006--177, CONICET, PIP 5650, UBACyT X058 and X108.
}



\date{\today}
\maketitle

\begin{abstract}
Let $G$ be an LCA group and $K$ a closed subgroup of $G$.
A closed subspace of $L^2(G)$ is called $K$-invariant if  it is invariant under translations by elements of $K$.

Assume now that $H$ is a countable uniform lattice in $G$ and $M$ is any closed subgroup of $G$ containing $H$.
In this article we study  necessary and sufficient conditions for an $H$-invariant space to be $M$-invariant.

As a consequence of our results we prove that for each closed subgroup $M$
of $G$ containing the lattice $H$, there exists an $H$-invariant space $S$ that is exactly $M$-invariant.
That is, $S$ is not invariant under any other subgroup $M'$ containing $M$. We also obtain estimates on the support of the  Fourier transform  of the generators of the $H$-invariant space, related to  its  $M$-invariance.
\end{abstract}

\section{Introduction}

Let $G$ be an LCA group and $K \subset G$ a  subgroup.
For $ y \in G$ let us denote by $t_y$ the translation operator acting on $ L^2(G)$. That is,
$t_y f(x) = f(x-y)$ for $x \in G$ and $f \in L^2(G)$.

A closed subspace $S$ of $L^2(G)$
satisfying that $t_k f \in S$ for every $ f\in S$ and every $k\in K$ is called $K$-invariant.

 In  case that  $G$ is $\R^d$ and $K$ is $\Z^d$ the subspace $S$ is called  shift-invariant.
Shift-invariant spaces are central in several areas such as approximation theory, wavelets,
frames and sampling.

The structure of these spaces for  the group $\R^d$ has been studied in \cite{Hel64,Sri64,HS64, dBDR94a,dBDR94b,Bow00} and, in
the context of general LCA groups, in \cite{KR08,CP09}.

Let now $H \subset G$ be a countable uniform lattice and $M$ be any closed subgroup of  $G$ satisfying that $H \subset M \subset G$.
In this article we want to study necessary and sufficient conditions in order that
an $H$-invariant space $S\subset L^2(G)$ is $M$-invariant.

For the case $(\Z,M,\R)$, this has been studied by Aldroubi et al in \cite{ACHKM09} and
for the case  $(\Z^d, M, \R^d)$ by Anastasio et al in \cite{ACP09}.

In this article we consider the extension to general LCA groups.
This is important in order to obtain general conditions that can be applied to different cases, as
for example  the case of the classic groups as the d-dimensional torus $\mathbb{T}^d$, the discrete group $\Z^d,$ and the finite group $\Z_d$.

This article is organized as follows. In Section \ref{sec-2} we set the notation and give some definitions. We study the properties of the invariance set  in Section \ref{sec-3}. In Section \ref{sec-4} we describe the structure  of $M$-invariant spaces and  range functions
in the context of LCA groups. The characterizations of $M$-invariance for $H$-invariant spaces is given in section \ref{sec-5-0}. Finally in Section \ref{sec-5} we give some applications.

\section{Notation}\label{sec-2}

Let $G$ be an arbitrary  locally compact Hausdorff abelian  group (LCA) written additively.
We will denote  by $m_G$ its Haar measure. The dual group of $G$, that is, the set of continuous characters on $G$,  is denoted  by $\widehat{G}$. The value of the character $\g\in\widehat{G}$ at the point $x\in G$, is written by $(x,\g)$.

The Fourier transform of a Haar integrable function $f$ on $G$, is the function $\widehat{f}$ on $\widehat{G}$ defined by
$$\widehat{f}(\g)=\int_G f(x)(x,-\g)\,dm_G(x),\quad\g\in\ \widehat{G}.$$
When the Haar measures $m_G$ and $m_{\widehat{G}}$ are normalized such that the Inversion Formula holds (see \cite{rudin}),
the Fourier transform on $L^1(G)\cap L^2(G)$ can be extended
to a unitary operator from $L^2(G)$ onto $L^2(\widehat{G})$, the so-called
Plancharel transformation. We also denote this transformation by ``$\land$".

Note that  the Fourier transform satisfies $\widehat{t_xf}(\cdot)=(-x,\cdot)\widehat{f}(\cdot)$.

For a subgroup $K$ of $G$, the set
$$K^*=\{\gamma \in \widehat{G}: (k,\g)=1, \,\,\forall\,\, k\in K\}$$
is called the {\it annihilator} of $K$. Since every character in $\widehat{G}$ is continuous, $K^*$ is a closed subgroup of $\widehat{G}$.

We will say that a closed subspace  $V\subset L^2(G)$ is {\it $K$-invariant}
if
$$ f\in V \Rightarrow t_kf\in V\quad\forall\,\,k\in K.$$

For a subset $\mathcal{A}\subset L^2(G)$, define
$$E_K(\A)= \{t_k\p: \p\in\A, k\in K\}\quad\textrm{and}\quad S_K(\A)=\clspan \,E_K(\A).$$
We call $S_K(\A)$
the $K$-invariant space generated by $\A$. If $\A=\{\p\}$,
we simply write $S_K(\p)$, and we call  $S_K(\p)$ a {\it principal  $K$-invariant space}.

Let $L$ be a subset of $G$, we will say that a function $f$ defined on $G$ is $L$-periodic
if $t_{\ell}f=f$ for all $\ell\in L$. A subset $B\subset G$ is $L$-periodic if its indicator function (denoted
by $\chi_B$) is $L$-periodic.

When two LCA groups $G_1$ and $G_2$ are topologically isomorphic we will write  $G_1\approx G_2$.

\section{The invariance  set}\label{sec-3}

Here and subsequently $G$ will be an LCA group and $H$ a countable uniform lattice in $G$. That is, a countable discrete subgroup of $G$ with compact quotient group $G/H$.

For simplicity of notation throughout this paper we will write $\G$ instead of $\widehat{G}$.

The aim of this work is  to characterize the extra invariance of an $H$-invariant space.
For this, given $S\subset L^2(G)$  an $H$-invariant space, we define the {\it invariance set } as
\begin{equation}\label{M}
M=\{x\in G\,:\, t_xf \in S, \,\forall\,\,f\in S\}.
\end{equation}
If  $\A$ is a set of generators for $S$, it is easy to check that $m\in M$ if and only if $t_m\p \in S$ for all $\p\in \A$.

In case that $M=G$,
Wiener's theorem (see \cite{Hel64}, \cite{Sri64} and \cite{HS64}) states that there exists
a measurable set $E\subset \G$
satisfying $$S=\{f\in L^2(G)\,:\, \supp(\widehat{f}\,)\subseteq E\}.$$
We want to describe $S$ when $M$ is not all $G$. We will first study the structure of the set $M$.

\begin{proposition}\label{M-isgroup}

Let $S$ be an $H$-invariant space of $L^2(G)$ and let $M$ be defined as in (\ref{M}).
Then $M$ is a closed subgroup of $G$ containing $H$.

\end{proposition}

For the proof of this proposition we will need the following lemma.
Recall that a semigroup is a nonempty  set with an associative additive operation.
\begin{lemma}\label{semi}
Let $K$ be a closed semigroup of $G$ containing $H$, then $K$ is a group.
\end{lemma}

\begin{proof}

Let $\pi$ be the quotient map from $G$ onto $G/H$.
Since  $K$ is a semigroup containing $H$, we have that $K+H=K$, thus
\begin{equation}\label{cerrado}
\pi^{-1}(\pi(K))=\bigcup_{k\in K}k+H=K+H=K.
 \end{equation}
This shows that $\pi(K)$ is closed in $G/H$ and therefore compact.

By \cite[Theorem 9.16]{HR79}, we have that a compact semigroup of $G/H$  is necessarily a group,
thus  $\pi(K)$ is a group and consequently $K$ is a group.

\end{proof}

\begin{proof}[\textit{ Proof of Proposition \ref{M-isgroup}}]

Since $S$ is an $H$-invariant space, $H\subset M$.

We first proceed to show that $M$ is closed.
Let $x_0\in G$ and $\{x_{\lambda}\}_{\lambda\in\Lambda}$ a net in $M$ converging to $x_0$.
Then
$$\lim_{\lambda}\|t_{x_{\lambda}}f-t_{x_0}f\|_2=0.$$
Since $S$ is closed, it follows that $t_{x_0}f\in S$, thus $x_0\in M$.

It is easy to check that $M$ is a semigroup of $G$, hence we conclude from Lemma \ref{semi}
that $M$ is a group.

\end{proof}

\section{The structure of principal $M$-invariant spaces}\label{sec-4}

\subsection{Preliminaries}\noindent
Shift-invariant spaces in $L^2(\R^d)$ are completely characterized using  fiberization techniques and range functions (see \cite{Bow00}). This theory has been extended to general LCA groups in \cite{CP09}. In what follows we state  some
definitions and properties given in that work.

We will assume that $G$ is a second countable LCA group and $H$ a countable uniform lattice in $G$.

The fact that $G$ is second countable, $G/H$ is compact and $\widehat{G/H}\approx H^*$,  implies that $H^*$ is  countable and discrete. Moreover, since  $\G/H^*\approx \widehat{H}$, $H^*$ is a countable uniform lattice in $\G$.
Therefore, there exists a measurable section $\W$ of $\G/H^*$ with finite $m_{\G}$-measure (see \cite{KK98}
and \cite{FG64}).

Let $L^2(\W,\ell^2(H^*))$
be the space of all measurable functions $\Phi:\W\rightarrow\ell^2(H^*)$ such that
$$\|\Phi\|_2^2:=\int_{\W}\|\Phi(\w)\|_{\ell^2(H^*)}^2\,dm_{\G}(\w)<\infty.$$

The following proposition shows that the space $L^2(\W,\ell^2(H^*))$ is isometric (up to a constant) to $L^2(G)$.

\begin{proposition}
The mapping $\T_H:L^2(G)\longrightarrow L^2(\W,\ell^2(H^*))$ defined as
$$\T_H f(\w)=\{\widehat{f}(\w+h^*)\}_{h^*\in H^*}$$
is an isomorphism that satisfies
$\|\T_H f\|_2=m_{H^*}(\{0\})^{1/2}\|f\|_{L^2(G)}.$
\label{prop-tau-H}
\end{proposition}

For $f\in L^2(G)$, the sequence $\T_H f(\w)=\{\widehat{f}(\w+h^*)\}_{h^*\in H^*}$ is {\it the $H$-fiber of $f$ at $\w$}.
Given a  subspace $V$ of $L^2(G)$ and $\w \in \W$,
the $H$-\emph{fiber space} of~$V$ at~$\w$ is
$$J_H(V)(\w)
=\overline{\{\T_H f(\w):
    f \in V }\},$$
where the closure is taken in the norm of $\ell^2(H^*)$.

The map that assigns to each $\w$ the fiber space $J_H(V)(\w)$
is known in the literature as the {\it range function} of $V$.

The following proposition characterizes  $H$-invariant spaces in terms of range functions and fibers.

\begin{proposition}\label{bownik-1-H}
If $S$ is an $H$-invariant space in $L^2(G)$
, then
$$S=\{f \in L^2(G) :
        \T_H f(\w)\in J_H(S)(\w) \text{ for a.e. }\w\in\W \}.$$
Moreover, if $S=S_H(\A)$ for a countable set $\A\subset L^2(G)$, then, for almost every $\w\in\W$,
$$J_H(S)(\w)=\clspan\{\T_H \varphi(\w):\, \varphi\in\A\} .$$
\end{proposition}

In this article we will use  fiberizations techniques for a  more general case, since the subspaces will be invariant under a closed subgroup which is not necessarily discrete.

The above results from \cite{CP09} can be extended straightforward to the case that the spaces are invariant under
a closed subgroup $M$ of $G$ containing a countable uniform lattice $H$ as follows.

Since $H\subset M$, we have that $M^*\subset H^*$ and, in particular, $M^*$ is discrete.
Thus, there exists  a countable section  $\SN$ of $H^*/M^*$.  Then, the set given by
\begin{equation}
\SR=\bigcup_{\s\in\SN} \W+\s
\end{equation}
is a $\sigma$-finite measurable section of the quotient $\G/M^*$. Using this section of $\G/M^*$ it is possible to obtain, in a way analogous to  the discrete case, the following.

\begin{proposition}\label{prop-tau-M}\noindent
\begin{enumerate}
\item [i)] \label{item-1-prop-tau}
The mapping $\T_M:L^2(G)\longrightarrow L^2(\SR,\ell^2(M^*))$ defined as
$$\T_M f(\dd)=\{\widehat{f}(\dd+m^*)\}_{m^*\in M^*}$$
is an isomorphism that satisfies
$\|\T_M f\|_2=m_{M^*}(\{0\})^{1/2}\|f\|_{L^2(G)}.$
\item [ii)]  \label{item-2-prop-tau}
Let $S$ be an $M$-invariant space generated by a countable set $\A$.
For each $\dd\in\SR$, define the \textnormal{$M$-fiber space} of $S$ at $\dd$ as
$$J_M(S)(\dd)=\clspan\{\T_M \varphi(\dd):\, \varphi\in\A\} .$$
 If $P$ and $P_{\dd}$ are the orthogonal projections onto $S$ and $J_M(S)(\dd)$ respectively, then, for every $g\in L^2(G)$,
$$\T_M(Pg)(\dd)=P_{\dd}(\T_Mg(\dd))\quad\text{a.e. }\dd\in\SR.$$
\item [iii)] \label{item-3-prop-tau}
If $S$ is an $M$-invariant space in $L^2(G)$, then
$$S=\{f \in L^2(G) :\T_M f(\dd)\in J_M(S)(\dd) \text{ for a.e. }\dd\in\SR \}.$$

\end{enumerate}
\end{proposition}

\subsection{Principal $M$-invariant Spaces}\noindent

We prove now the following  characterization of principal $M$-invariant spaces. This result extend the $\R^d$ case.

\begin{theorem}
Let $f\in L^2(G)$ and $M$ a closed subgroup of $G$ containing $H$.
If $g\in S_M(f)$, then there exists an $M^*$-periodic function
$\eta$ such that $\widehat{g}=\eta\widehat{f}$.

Conversely, if $\eta$ is an $M^*$-periodic function such that $\eta\widehat{f}\in L^2(\G)$, then the
function $g$ defined by $\widehat{g}=\eta\widehat{f}$ belongs to $S_M(f)$.
\label{rango-SIS-M-2}
\end{theorem}

\begin{proof}
Let us call $S=S_M(f)$ and let $P$ and $P_{\dd}$ be the orthogonal projections onto $S$ and $J_M(S)(\dd)$ respectively. Given
$g\in S$, we first define $\eta_g$ in $\SR$ as
$$\eta_g(\dd)=
\begin{cases}
\frac{\langle \T_M g(\dd), \T_M f(\dd)\rangle}{\|\T_M f(\dd)\|_2^2} & \textrm{if}\,\, \dd\in\ E_f\\
0 & \textrm{otherwise},
\end{cases}
$$
where $E_f$ is the set $\{ \dd\in\SR\,:\, \|\T_M f(\dd)\|_2^2\neq 0\}$.  Then, since $\{\SR+m^*\}_{m^*\in M^*}$ forms a partition of $\G$, we can  extend $\eta_g$ to all $\G$ in an $M^*$-periodic way.

Now, by Proposition \ref{prop-tau-M} we have that
$$\T_M g(\dd)=\T_M (Pg)(\dd)=P_{\dd}(\T_Mg(\dd))=\eta_g(\dd)\T_Mf(\dd).$$

Since $\eta_g$ is an $M^*$-periodic function, $\widehat{g}=\eta_g\widehat{f}$ as we wanted to prove.

Conversely, if $\widehat{g}=\eta\widehat{f}$, with $\eta$ an $M^*$-periodic function, then $\T_M g(\dd)=\eta(\dd)\T_M f(\dd)$. By Proposition \ref{prop-tau-M},
 $g\in S$.
\end{proof}

\section{Characterization of  M-invariance}\label{sec-5-0}

If $H\subset M\subset G$, where $H$ is a countable uniform lattice in $G$ and $M$ is a closed subgroup of $G$,  we are interested in describing  when an $H$-invariant space $S$ is also $M$-invariant.

Let $\W$ be a measurable section of $\G/H^*$ and $\SN$ a countable section of $H^*/M^*$.
For $\s\in\SN$
we define the set  $B_{\s}$ as
\begin{equation}\label{B-sigma}
B_{\s}=\W+\s+M^*=\bigcup_{m^*\in M^*}(\W+\s)+m^*.
\end{equation}
Therefore, each $B_{\s}$ is an $M^*$-periodic set.

Since $\W$ tiles $\G$ by $H^*$ translations and $\SN$ tiles $H^*$ by $M^*$ translations, it follows that $\{B_{\s}\}_{\s\in\SN}$ is a partition of $\G$.

Now, given an $H$-invariant space $S$,  for each $\s\in \SN$, we define  the subspaces
\begin{equation}\label{U-sigma}
U_{\s}=\{f\in L^2(G):\widehat{f}=\chi_{B_{\s}}\widehat{g},
\,\,\textrm{with} \,\,g\in S\}.
\end{equation}

\subsection{Characterization of $M$-invariance in terms of subspaces}\noindent

The main theorem of this section characterizes the $M$-invariance of $S$ in terms of
the subspaces $U_{\s}$.

\begin{theorem}\label{teo-subs}
If $S\subseteq L^2(G)$ is an $H$-invariant space and $M$ is a closed subgroup of $G$ containing $H$, then the following are equivalent.
\begin{enumerate}
\item [i)]$S$ is $M$-invariant.
\item [ii)]$U_{\s}\subseteq S$ for all $\s\in \SN$.
\end{enumerate}
Moreover, in case any of these hold we have that $S$ is the orthogonal direct sum
$$S=\dotbigoplus_{\s\in \SN} U_{\s}.$$
\end{theorem}

Now we state a lemma that we need  to prove Theorem \ref{teo-subs}.

\begin{lemma}\label{U}
Let $S$ be an $H$-invariant space and  $\s\in\SN$. Assume that the subspace $U_{\s}$ defined in (\ref{U-sigma}) satisfies $U_{\s}\subseteq S$. Then, $U_{\s}$ is an $M$-invariant space and in particular is $H$-invariant.
\end{lemma}

\begin{proof}
Let us prove first that $U_{\s}$ is closed.
Suppose that $f_j\in U_{\s}$ and $f_j\rightarrow f$ in $L^2(G)$.
Since $U_{\s}\subseteq S$  and $S$ is closed, $f$ must be in $S$.
Further,
$$\|\widehat{f_j}-\widehat{f}\|^2_2=\|(\widehat{f_j}-\widehat{f})\chi_{B_{\s}}\|^2_2+\|(\widehat{f_j}-\widehat{f})\chi_{B_{\s}^c}\|^2_2
	=\|\widehat{f_j}-\widehat{f}\chi_{B_{\s}}\|^2_2+\|\widehat{f}\chi_{B_{\s}^c}\|^2_2.$$

Since the left-hand side converges to zero, we must have that $\widehat{f}\chi_{B_{\s}^c}=0$ a.e.
$\g\in\G$. Then,
$\widehat{f}=\widehat{f}\chi_{B_{\s}}.$
Consequently $f\in U_{\s}$, so $U_{\s}$ is closed.

Now we show that $U_{\s}$ is $M$-invariant.
Given $m\in M$ and $f\in U_{\s}$, we will prove that
$  (m,\cdot)\widehat{f}(\cdot)\in \widehat{U}_{\s}.$

Since $f\in U_{\s}$,
there exists $g\in S$ such that $\widehat{f}=\chi_{B_{\s}}\widehat{g}$.
Hence,
\begin{equation}\label{ecu-1-lema-U}
(m,\cdot)\widehat{f}(\cdot)=(m,\cdot)(\chi_{B_{\s}}\widehat{g})(\cdot)=\chi_{B_{\s}}(\cdot)((m,\cdot)\widehat{g}(\cdot)).
\end{equation}
If we were able to find an $H^*$-periodic function $\ell_m$ verifying
\begin{equation}\label{ecu-2-lema-U}
(m,\g)=\ell_m(\g)\quad\textrm{a.e.}\,\g\in B_{\s},
\end{equation}
then, we can rewrite (\ref{ecu-1-lema-U}) as
$$(m,\cdot)\widehat{f}(\cdot)=\chi_{B_{\s}}(\cdot)(\ell_m\widehat{g})(\cdot).$$
Theorem \ref{rango-SIS-M-2} can then be applied for the uniform lattice $H$. Thus, since $\ell_m$ is $H^*$-periodic, we obtain that $\ell_m\widehat{g}\in \widehat{S_H(g)}\subset\widehat{S}$
and so, $(m,\cdot)\widehat{f}(\cdot)\in \widehat{U}_{\s}.$

Now we define the function $\ell_m$ as follows.
For each $h^*\in H^*$, set
\begin{equation}\label{ecu-4-lema-U}
\ell_m(\w+h^*)=(m,\w+\s)\quad\textrm{a.e.}\,\,\w\in\W.
\end{equation}
It is clear that $\ell_m$ is $H^*$-periodic.

Since $(m, \cdot)$ is $M^*$-periodic,
\begin{equation*}
(m,\w+\s)=(m,\w+\s+m^*)\quad\textrm{a.e.}\,\,\w\in\W,\, \forall\,m^*\in M^*.
\end{equation*}
Thus, (\ref{ecu-2-lema-U}) holds.

Note that, since $H\subset M$, the $H$-invariance of $U_{\s}$ is a consequence of the $M$-invariance.

\end{proof}

\begin{proof}[Proof of Theorem \ref{teo-subs}]
i)\,$\Rightarrow$\,ii): Fix $\s\in \SN$ and $f\in U_{\s}$. Then $\widehat{f}=\chi_{B_{\s}}\widehat{g}$
for some $g\in S.$ Since $\chi_{B_{\s}}$ is an $M^*$-periodic function, by Theorem \ref{rango-SIS-M-2}, we have that $f\in S_M(g)\subseteq S$, as we wanted to prove.

ii)\,$\Rightarrow$\,i): Suppose that $U_{\s}\subseteq S$ for all $\s\in \SN$.
Note that Lemma \ref{U} implies that
$U_{\s}$ is $M$-invariant, and we also have that the $U_{\s}$ are mutually orthogonal
since the sets $B_{\s}$ are disjoint.

Suppose that $f\in S$. Then, since $\{B_{\s}\}_{\s\in \SN}$ is a partition of $\G$, it follows that
$\widehat{f}=\sum_{\s\in\SN}\widehat{f}\chi_{B_{\s}}$.
This implies that $f\in \dotbigoplus_{\s\in \SN} U_{\s}$ and
consequently, $S$ is the orthogonal direct sum
$$S=\dotbigoplus_{\s\in \SN} U_{\s}.$$
As each $U_{\s}$ is $M$-invariant, so is $S$.
\end{proof}

\subsection{Characterization of $M$-invariance in terms of $H$-fibers}\noindent

In this section we will first express the conditions of  Theorem \ref{teo-subs} in terms of $H$-fibers.
Then, we will give a useful characterization of the $M$-invariance for a finitely generated  $H$-invariant space in terms of the Gramian.

If $f\in L^2(G)$ and $\s\in \SN$, we define the function $f^{\s}$ by
$$\widehat{f^{\s}}=\widehat{f}\chi_{B_{\s}}.$$

Let $P_{\s}$ be the orthogonal projection onto $S_{\s}$, where
$$S_{\s}=\{f\in L^2(G)\,:\,\text{ supp}(\widehat{f})\subset B_{\s}\}.$$
Therefore
$$f^{\s}=P_{\s}f\quad\text{and}\quad U_{\s}=P_{\s}(S)=\{f^{\s}\,:\,f\in S\}.$$
Moreover, if $S=S_H(\A)$ with $\A$ a countable subset of $L^2(G)$, then
\begin{equation}\label{JUsigma}
J_H(U_{\s})(\w)=\overline{\textnormal{span}}\{\T_H(\varphi^{\s})(\w)\,:\,\varphi\in\A\}.
\end{equation}

\begin{remark}\label{fibra-cut-off}
Note that the fibers
$$\T_H(\varphi^{\s})(\w)=\{\chi_{B_{\s}}(\w+h^*)\widehat{\varphi}(\w+h^*)\}_{h^*\in H^*}$$
can be described in a simple way as
\begin{equation*}
\chi_{B_{\s}}(\w+h^*)\widehat{\p}(\w+h^*)=
\begin{cases}
\widehat{\p}(\w+h^*)&\text{ if }h^*\in\s+M^*\\
0&\text{ otherwise }.
\end{cases}
\end{equation*}
Therefore,  if $\s\neq\s'$, $J_H(U_{\s})(\w)$ and $J_H(U_{\s'})(\w)$ are  orthogonal subspaces for a.e. $\w\in\W$.
\end{remark}

Combining Theorem~\ref{teo-subs} with Proposition~\ref{bownik-1-H} and (\ref{JUsigma})
we obtain the following result.
\begin{proposition}\label{fibras-en-S}
Let $S$ be an $H$-invariant space  generated by a countable set $\A\subset L^2(G)$. The following statements are equivalent.
\begin{enumerate}
\item[i)]
$S$ is $M$-invariant.
\item[ii)]
$\T_H(\varphi^{\s})(\w)\in J_H(S)(\w)$ a.e. $\w\in\W$
for all  $\varphi\in\A$ and $\s\in\SN$.
\end{enumerate}
\end{proposition}

Let $\Phi= \set{\varphi_1,\dots,\varphi_{\ell}}$ be a finite collection of  functions in $L^2(G)$.
Then, the \emph{Gramian} $G_\Phi$ of $\Phi$ is the
$\ell \times \ell$ matrix of $H^*$-periodic functions
\begin{align}
[G_{\Phi}(\omega)]_{ij}
&=
\Big\langle \T_H(\varphi_i)(\w), \T_H(\varphi_j)(\w)\Big\rangle\notag\\
&=m_{H^*}(\{0\}) \sum_{h^* \in H^*} \widehat{\varphi_i}(\omega+h^*) \, \overline{\widehat{\varphi_j}(\omega+h^*)}
\label{gram}
\end{align}
for $ \omega\in \W$.

Given a subspace $V$ of $L^2(G)$, the  {\it dimension function} is defined by
$$\dim_V:\W\rightarrow \N_0\cup\{\infty\}, \quad \dim_V(\w):=\dim(J_H(V)(\w)).$$

We will also need the next result which is a straightforward consequence of Proposition  \ref{prop-tau-H}  and Proposition \ref{bownik-1-H}.

\begin{proposition} \label{ortho}
Let $S_1$ and $S_2$ be $H$-invariant spaces.
If $S = S_1 \dotoplus S_2$, then
$$J_H(S)(\omega)
\EQ J_H(S_1)(\omega) \dotoplus J_H(S_2)(\omega), \quad\text{a.e. }\w\in\W.$$
\end{proposition}
The converse of this proposition is also true, but will not be needed.

Now we give a slightly simpler
characterization of $M$-invariance for the finitely generated case.

\begin{theorem} \label{fibers}
If $S$ is an $H$-invariant space, finitely generated by $\Phi$, then
the following statements are equivalent.

\begin{enumerate}
\item[i)]
$S$ is $M$-invariant.
\smallskip
\item[ii)]
For almost every $\omega \in \W$,
$\dim_S(\omega)
= \sum_{\s\in\SN} \dim_{U_{\s}}(\omega).$
\smallskip
\item[iii)]
For almost every $\omega \in \W$,
$\textnormal{rank}[G_{\Phi}(\w)]=\sum_{\s\in\SN}\textnormal{rank}[G_{\Phi^{\s}}(\w)],$
where $\Phi^{\s}=\{\varphi^{\s}\,:\,\varphi\in\Phi\}.$
\end{enumerate}
\end{theorem}

\begin{proof}
i)\,$\Rightarrow$\,ii):
By Theorem \ref{teo-subs}, $S=\dotoplus_{\s\in\SN} U_{\s}$.
Then, ii) follows from Proposition \ref{ortho}.

ii)\,$\Rightarrow$\,i):
Since $\{B_{\s}\}_{\s\in\SN}$ is a partition of $\G$, $S\subset \dotoplus_{\s\in\SN} U_{\s}$.
Then, by  Remark \ref{fibra-cut-off} we have that
$$J_H(S)(\w)\subset \dotoplus_{\s\in\SN} J_H(U_{\s})(\w).$$
Using ii), we obtain that $J_H(S)(\w)=\dotoplus_{\s\in\SN} J_H(U_{\s})(\w)$.
The proof follows as a consequence of Proposition \ref{fibras-en-S}.

The equivalence between ii) and iii) follows from (\ref{JUsigma}).

\end{proof}

\section{Applications of $M$-invariance}\label{sec-5}

In this section we estimate the size of the supports of the Fourier transforms of the generators
of a finitely generated $H$-invariant space which is also $M$-invariant.

We will not include the proofs of the results stated bellow,  since they follow readily from the
$\R^d$ case (see \cite[Section 6]{ACP09}).

As a consequence of Theorem \ref{fibers} we  obtain the following theorem.

\begin{theorem}\label{soporte}
Let $S$ be an $H$-invariant space,
finitely generated  by the set  $\{\varphi_1,\ldots,\p_{\ell}\}$, and define
$$ E_j=\{\w\in\W\,: \, \dim _S(\w)=j\},\quad j=0,\ldots, \ell.$$
If $S$ is $M$-invariant and $\SR'$ is any measurable section of $\G/M^*$, then
$$m_{\G}\big(\{ y\in \SR'\,: \, \widehat{\p_i}(y)\neq 0\}\big)\leq \sum_{j=0}^\ell m_{\G}(E_j)\,j\leq m_{\G}(\W)\,\ell,$$
for each $i=1,\ldots, \ell$.
\end{theorem}

\begin{corollary}\label{vanish}
Let $\varphi\in L^2(G)$ be given. If $S_H(\varphi)$ is $M$-invariant for some closed subgroup $M$ of $G$ such that $H \subsetneqq M$, then $\widehat{\varphi}$ must vanish on a set of infinite $m_{\G}$-measure.
\end{corollary}

Since the uncertainty principle holds in LCA groups (see \cite{Smi90}), we also can extend Proposition 6.4 of \cite{ACP09} as follows.

\begin{proposition}
If a nonzero function $\varphi\in L^2(G)$ has compact support, then $S_H(\varphi)$
is not $M$-invariant for any $M$ closed subgroup of $G$ such that $H\subsetneqq M$.
In particular, if $\,x\in G$ and $  x\notin H$,
\begin{equation*}
t_x\p\notin S_H(\varphi).
\end{equation*}
\end{proposition}

As a consequence of Theorem \ref{soporte}, in  case that $M=G$, we obtain the following corollary.

\begin{corollary}
If $\varphi\in L^2(G)$ and $S_H(\varphi)$ is $G$-invariant, then $$m_{\G}\big(\textnormal{supp}{(\widehat{\varphi})}\big)\leq m_{\G}(\W).$$
\end{corollary}

\subsection{Exactly $M$-invariance}\noindent

Let $M$ be a closed subgroup of $G$ containing a countable uniform lattice $H$. The next theorem states that there exists an M-invariant space $S$ that is {\it not} invariant under any  vector outside $M$. We will say in this case that $S$ is {\it exactly} $M$-invariant.

Note that because of Proposition  \ref{M-isgroup},  an M-invariant space is exactly $M$-invariant if and only if it is not invariant under any closed subgroup $M'$ containing $M.$

\begin{theorem}
For each closed subgroup $M$ of  $G$ containing a countable uniform lattice $H$, there exists a shift-invariant space of $L^2(G)$
which is exactly $M$-invariant.
\end{theorem}
\begin{proof}

Suppose that  $0\in\SN$ and take $\p\in L^2(G)$ satisfying $\supp(\widehat{\p})=B_0$, where $B_0$ is defined as in (\ref{B-sigma}). Let $S=S(\varphi)$.

Then, $U_0=S$ and $U_{\s}=\{0\}$ for $\s\in\SN$, $\s\neq 0$. So,  as a consequence of Theorem \ref{teo-subs}, it follows that $S$ is $M$-invariant.

Now, if $M'$ is a closed subgroup such that $M\subsetneqq M'$, we will show that $S$ can not be $M'$-invariant.

Since $M\subset M'$, $(M')^*\subset M^*$. Consider  a section $\mathcal{C}$ of the quotient $M^*/(M')^*$ containing the origin. Then,
the set given by $$\SN':=\{\s+c\,:\,\s\in\SN,\, c\in \mathcal{C}\},$$
is a section of $H^*/(M')^*$ and $0\in\SN'$.

If $\{B'_{\s'}\}_{\s'\in\SN'}$ is the partition defined in (\ref{B-sigma}) associated to $M'$, for each $\s\in\SN$
it holds that $\{B'_{\s+c}\}_{c\in\mathcal{C}}$ is a partition of $B_{\s}$, since
\begin{equation}\label{B'parteB}
B_{\s}=\W+\s+M^*=\bigcup_{c\in\mathcal{C}}\W+\s+c+(M')^*=\bigcup_{c\in\mathcal{C}}B'_{\s+c}.
\end{equation}

We will show now that $U'_0\nsubseteq S,$ where $U'_0$ is the subspace defined in (\ref{U-sigma}) for $M'$.
Let $g\in L^2(G)$ such that $\widehat{g}=\widehat{\p}\chi_{B'_0}$. Then $g\in U'_0.$
Moreover, since $\supp(\widehat{\p})=B_0$, by (\ref{B'parteB}), $\widehat{g}\neq 0$.

Suppose that $g\in S$, then $\widehat{g}=\eta\widehat{\p}$ where $\eta$ is an $H^*$-periodic function.
Since $M\subsetneqq M'$, there exists $c\in\mathcal{C}$ such that $c\neq 0$.
By (\ref{B'parteB}), $\widehat{g}$ vanishes in $B'_c$.  Then, the $H^*$-periodicity of $\eta$ implies that $\eta(\g)=0$ a.e. $\g\in\G$. So $\widehat{g}=0$, which is a contradiction.

This shows that $U'_0\nsubseteq S$. Therefore, $S$ is not $M'$-invariant.
\end{proof}

\thispagestyle{empty}

\end{document}